\documentclass{amsart}

\usepackage{amssymb}
\usepackage{latexsym}
\usepackage{amsmath}
\usepackage{euscript}
\usepackage{graphics}
\usepackage[active]{srcltx}

\usepackage{graphicx}
\newcommand*{\Scale}[2][4]{\scalebox{#1}{$#2$}}%

      \def\dC{{\mathbb C}}
\def\dD{{\mathbb D}}

      \def\dR{{\mathbb R}}
   \def\dT{{\mathbb T}}

\def\cA{{\mathcal A}}   \def\cB{{\mathcal B}}   \def\cC{{\mathcal C}}
\def\cD{{\mathcal D}}

   \def\cW{{\mathcal W}}

\def\Im{\operatorname{Im}}
\DeclareMathOperator{\Rl}{Re}

\def\downbar#1{
\setbox10=\hbox{$#1$}
   \dimen10=\ht10 \advance\dimen10 by 2.5pt
   \ifdim \dimen10<15pt 
      \advance\dimen10 by -0.5pt
      \dimen11=\dimen10
      \advance\dimen10 by 2.5pt
      \lower \dimen11
   \else \lower \ht10 \fi
   \hbox {\hskip 1.5pt \vrule height \dimen10 depth \dp10}\relax}
 \def\upbar#1{
 \setbox10=\hbox{$#1$}
    \dimen10=\ht10 \advance\dimen10 by \dp10 \advance\dimen10 by 2.5pt
    \ifdim \dimen10<15pt 
       \advance\dimen10 by 2pt \fi
    \raise 2.5pt \hbox {\hskip -1.5pt \vrule height \dimen10}\relax}
\def\cfr#1#2{
 \downbar{#2} \hskip -1.5pt {\; #1 \; \over \thinspace \  #2}\upbar{#1}}

\newtheorem{theorem}{Theorem}[section]
\newtheorem{proposition}[theorem]{Proposition}
\newtheorem{corollary}[theorem]{Corollary}

\theoremstyle{definition}

\newtheorem{remark}[theorem]{Remark}
\newtheorem*{conjecture}{Conjecture}

\numberwithin{equation}{section}

\begin{document}
\title[Wall's continued fractions and Jacobi matrices]
{A note on Wall's modification of the Schur algorithm and linear pencils of Jacobi matrices}

\author{Maxim  Derevyagin}
\address{
Maxim Derevyagin\\
University of Mississippi\\
Department of Mathematics\\
Hume Hall 305 \\
P. O. Box 1848 \\
University, MS 38677-1848, USA }
\email{derevyagin.m@gmail.com}

\date{\today}


\begin{abstract}
In this note  we revive a transformation that was introduced by H. S. Wall and that establishes a one-to-one correspondence between continued fraction representations of Schur, Carath\'eodory, and Nevanlinna functions. This transformation can be considered as an analog of the Szeg\H{o} mapping but it is based on the Cayley transform, which relates the upper half-plane to the unit disc. For example, it will be shown that, when applying the Wall transformation, instead of OPRL, we get a sequence of orthogonal rational functions that satisfy three-term recurrence relation of the form $(H-\lambda J)u=0$, where $u$ is a semi-infinite vector, whose entries are the rational functions. Besides, $J$ and $H$ are Hermitian Jacobi matrices for which a version of the Denisov-Rakhmanov theorem holds true. Finally we will demonstrate how pseudo-Jacobi polynomials (aka Routh-Romanovski polynomials) fit into the picture.  
\end{abstract}

\maketitle

\section{Introduction}

In September of the year 1916, Issai Schur submitted the first paper of the series of two \cite{Schur1}, \cite{Schur2} that presented a new parametrization of functions that are analytic and bounded by 1 in the open unit disc $\dD$ and an algorithm for computing the corresponding parameters. The algorithm is now known as the Schur algorithm. In fact, it's been literally a hundred years, and yet there is still a continuing interest in further developing the findings of I. Schur. One of the reasons for that is because the Schur algorithm is a successor of the Euclidean algorithm, which has many theoretical and practical applications. Another one is that the Schur algorithm is intimately related to orthogonal polynomials on the unit circle (hereafter abbreviated by OPUC) and the latter has seen an enormous progress since the beginning of the 21st century. However, the ideology of this note is based on an old result that appeared in 1944 in a paper by Hubert Stanely Wall \cite{Wall44}. Nevertheless, a proper recasting of the result gives new insights and perspectives to the theory of orthogonal polynomials. We will see it later but now let's briefly recall basics of the Schur algorithm. First of all, we need to consider a Schur function $f$, which is an analytic function mapping $\dD$ to its closure $\overline{\dD}$, that is, 
\begin{equation*} 
\sup_{z\in\dD}\, |f(z)| \leq 1.
\end{equation*}   
As a matter fact, a Schur function is the input for the Schur algorithm. Namely, given a Schur function $f$ the Schur algorithm generates a sequence of Schur functions $\{f_n\}_{n=0}^{\infty}$ by means of the following relations
\begin{equation} \label{SchurAlg} 
\begin{split}
f_0(z)&=f(z), \\
 f_n(z) &= \frac{\gamma_n + zf_{n+1}(z)}{1+\bar\gamma_n zf_{n+1}(z)}, \quad n=0, 1, 2, \dots,
\end{split}
\end{equation}  
where $\gamma_n=f_n(0)$ are called Schur parameters and satisfy the relation
\begin{equation}\label{SchurCond}
|\gamma_n|<1, \quad n=0, 1, 2, \dots.
\end{equation}
To be more precise, the Schur algorithm gives a sequence that is either finite or infinite. In what follows we are mainly interested in Schur functions that produce infinite sequences of Schur parameters. In other words, when we say that $f$ is a Schur function we mean that there is an infinite sequence of linear fractional transformations \eqref{SchurAlg} generated by $f$ unless said otherwise. 

The Schur algorithm is the key to Wall's theory and we are now ready to proceed with it, which we are going to do in the following way. The goal of the next section is to show that orthogonal polynomials on the real line (henceforth OPRL) and OPUC correspond to totally different interpolation problems. Hence, in addition to the attempt to identify OPUC and OPRL, it is also natural to find the real line image of OPUC when applying a transformation that keeps the underlying interpolation problems equivalent. This is what is actually done in Section 3 and Section 4 through the findings of H. S. Wall. The section after that is where we simply adopt a more general theory developed in \cite{BDZh}, \cite{Der10}, and \cite{DZh} to this very particular case of Wall's continued fraction representation of Nevanlinna functions. Besides, Section 5 reveals the relation between the spectral theories of OPUC and linear pencils of Jacobi matrices. At the end, in Section 6, we give an example based on pseudo-Jacobi polynomials, which are sometimes called Romanovski or even Routh-Romanovski polynomials.

\section{The same old moment problems}

One of the original ideas to construct the Schur algorithm was to solve an interpolation problem, which now bears the name Schur's coefficient problem (for instance see \cite[Chapter 3, Section 3]{A65} or \cite[Section 9]{DK03}). This interpolation problem is not of principal interest here and it's better for us to consider an equivalent problem in the class of Carath\'eodory functions (see \cite[Chapter 5, Section 1]{A65}).  
Before formulating it, recall that a Carath\'eodory function $F$ is an analytic function on $\dD$ which obeys
\begin{equation} \label{CarCond} 
F(0)=1, \qquad  \Rl F(z) >0\,\mbox{ when }\,z\in\mathbb{D}. 
\end{equation}
It is well known that if $f$ is a Schur function then the function
\begin{equation} \label{Carathe} 
F(z) = \frac{1+zf(z)}{1-zf(z)}
\end{equation}
is a Carth\'eodory function and vice versa.
Indeed, for the function $F$ defined by \eqref{Carathe} it is easily seen that 
\begin{equation}\label{ReOfCar}
\Rl F(z)=\frac{1-|zf(z)|^2}{|1-zf(z)|^2}>0, \quad z\in\dD. 
\end{equation} 
So, we are in the position to formulate the Carath\'eodory coefficient problem: given the complex numbers $c_1$, $c_2$, $c_3$, \dots, find necessary and sufficient conditions for the function 
\begin{equation}\label{CarCP}
F(z)=1+2c_1z+2c_2z^2+2c_3z^3+\dots
\end{equation}
to be a Carath\'eodory function. To get to the conditions using the Schur algorithm is relatively easy since the Schur parameters can be expressed in terms of the coefficients $c_1$, $c_2$, $c_3$, \dots. For instance, the formulas can be extracted in a similar way as it is done in \cite[Section 1.3]{OPUC1} (see also \cite[Section 9]{DK03}). Therefore, the condition \eqref{SchurCond} is a solution to the the Carath\'eodory coefficient problem. Clearly, there could be only one function corresponding to the coefficients $c_1$, $c_2$, $c_3$, $\dots$ due to  the uniqueness theorem for analytic functions. This means that we don't have to think about describing all possible functions generated by the given sequence $c_1$, $c_2$, $c_3$, \dots. Therefore, the problem is fully resolved.

Another way of solving Carath\'eodory's coefficient problem leads to trigonometric moment problems. To see this, one needs to take into account that Carath\'{e}odory functions admit the representation \cite[Chapter 3, Section 1]{A65}
\begin{equation} \label{IntCara} 
F(z) = \int_0^{2\pi} \frac{e^{i\theta}+z}{e^{i\theta}-z}\, d\mu(\theta)
\end{equation} 
for some non-trivial (that is, infinitely supported) probability measure $\mu$. Next, after combining \eqref{CarCP} and \eqref{IntCara} one can see that the Carath\'eodory coefficient problem reads: find necessary and sufficient conditions on the coefficients $c_1$, $c_2$, $c_3$, $\dots$ in order that there exists a probability measure $\mu$ such that
\[
c_n=\int_0^{2\pi} e^{-in\theta}\, d\mu(\theta), \quad n=1,2,3,\dots.
\]
Once we have a moment problem it seems natural to consider orthogonal polynomials since they play a prominent role for the analogous moment problems on the real line, which will be discussed later. So, given a moment sequence $c_1$, $c_2$, $c_3$, $\dots$ one can define a sequence of monic polynomials by the formulas
\[
\Phi_n(z)=\frac{1}{D_{n-1}}\begin{vmatrix}
c_0&\overline{c}_1&\dots&\overline{c}_n\\
c_1&{c}_0&\dots&\overline{c}_{n-1}\\
\vdots&\vdots&&\vdots\\
c_{n-1}&{c}_{n-2}&\dots&\overline{c}_1\\
1&z&\dots&z^n\\
\end{vmatrix}, \quad n=0,1,2\dots,
\]
where $c_0=1$, $D_{-1}=1$, and $D_{n-1}=\det(c_{k-j})_{k,j=0}^{n-1}$ with $c_{-j}=\overline{c}_j$ for $j=1,2,3\dots$. Evidently, $\Phi_n$ is correctly defined if and only if $D_{n-1}\ne 0$, which is a step towards the condition we are looking for. 
In addition, if $D_{n-1}\ne 0$ for $n=1,2,\dots$ then the polynomials $\Phi_n$ satisfy the Szeg\H{o} recurrence \cite{Ger40}:
\begin{equation}\label{SzRec}
\begin{split}
\Phi_{n+1}(z)=z\Phi_n(z)-\overline{\alpha}_n\Phi_n^*(z)\\
\Phi_{n+1}^*(z)=\Phi_n^*(z)-\alpha_nz\Phi_n(z),
\end{split}
\end{equation}
where $\Phi_0=1$ and $\Phi_n^*$ is the polynomial reversed to $\Phi_n$, that is,
\begin{equation}\label{RevPol}
\Phi_n^*(z)=z^n\overline{\Phi_n(1/\overline{z})}.
\end{equation}
Finally, according to Favard's theorem on the unit circle (for example, see \cite{ENZG91}) the polynomials $\Phi_n$ are orthogonal with respect to a positive measure supported on the unit circle $\dT$ if and only if the coefficients $\alpha_n$ called Verblunsky coefficients satisfy
\begin{equation}\label{VerbCond}
|\alpha_n|<1, \quad n=0, 1, 2, \dots,
\end{equation}
which delivers the condition we wanted to get to resolve the Carath\'eodory coefficient problem. The latter condition came into play in a way different from using the Schur algorithm but \eqref{VerbCond} is actually the same as \eqref{SchurCond}  if we take into account the Geronimus theorem \cite[Theorem 3.1.4]{OPUC1}:
\[
\alpha_n=\gamma_n,\quad n=0,1,2,\dots.
\]
Thus, we are back to the Schur algorithm. As a result, one sees that OPUC are associated with solving the interpolation problem in the class of Carath\'eodory functions and, consequently, they absorb the information about the interpolation.

To close up this discussion we should consider the real line case that lies within the same circle of concepts. Let us begin with a Hamburger moment problem \cite[Chapter 2, Section 1]{A65}: given an infinite sequence of real numbers $s_0=1$, $s_1$, $s_2$, \dots; it is required to find a probability measure $\sigma$ supported on the real line $\dR$ such that
\begin{equation}\label{Hmp}
s_n=\int_{\dR} t^{n}\, d\sigma(t), \quad n=1,2,3,\dots.
\end{equation}
This problem is not always solvable and therefore we need to discuss the existence criterion, which will be done through the use of OPRL. To this end, introduce the polynomials
\[
P_n(\lambda)=\frac{1}{\sqrt{|\Delta_n\Delta_{n-1}|}}\begin{vmatrix}
s_0&{s}_1&\dots&{s}_n\\
s_1&{s}_2&\dots&{s}_{n+1}\\
\vdots&\vdots&&\vdots\\
s_{n-1}&{s}_{n}&\dots&s_{2n-1}\\
1&\lambda&\dots&\lambda^n\\
\end{vmatrix}, \quad n=0,1,2\dots,
\]
where $\lambda$ is in the upper half-plane $\dC_+$, $\Delta_{-1}=1$, and $\Delta_n=\det(s_{k+j})_{k,j=0}^{n}$. These polynomials are correctly defined provided that $\Delta_n\ne0$ for $n=1,2,3,\dots$ and, in this case, are orthogonal with respect to a quasi-definite moment functional, which implies that they satisfy three-terms recurrence relations \cite[Chapter 1]{Chi78}:
\begin{equation}\label{qOPrelations}
\lambda P_{n}(\lambda)=b_{n}P_{n+1}(\lambda)+a_nP_n(\lambda)+\epsilon_{n-1}b_{n-1}P_{n-1}(\lambda), \quad n=0,1,2,\dots
\end{equation}  
where $b_{-1}=0$, $\epsilon=\pm 1$, $b_n>0$ and $a_n\in\dR$. Then in this context the Favard theorem reads that there exists a measure $\sigma$ satisfying \eqref{Hmp} if and only if $\epsilon_n=1$ for all nonnegative integers $n$ \cite[Chapter I, Theorem 4.4]{Chi78}. While we are on the subject, it is worth mentioning that in the latter case \eqref{qOPrelations} can be rewritten by means of a symmetric tridiagonal matrix called a Jacobi matrix
\[
\begin{pmatrix}
  a_0 & b_0&  &\\
  b_0& a_1& b_1&\\
     & b_1&a_2&\\
     &&&\ddots&    
      \end{pmatrix}
      \begin{pmatrix}
      P_0\\
      P_1\\
      P_2\\
      \vdots
      \end{pmatrix}=\lambda \begin{pmatrix}
      P_0\\
      P_1\\
      P_2\\
      \vdots
      \end{pmatrix}.
\]
It is noteworthy that such an explicit appearance of Jacobi matrices here is strikingly different to the unit circle case for which it took many years and papers to create a proper analog of Jacobi matrices (for details see \cite[Chapter 4]{OPUC1}).  

What concerns the underlying interpolation problems, one has to recall that a Nevanlinna function $\varphi$ is an analytic function on $\dC_+$ which satisfies
\[
\Im \varphi(\lambda)>0, \quad \lambda\in\dC_+.
\]
Next, it is not so hard to check that if $\sigma$ is a positive measure on $\dR$ then the function 
\[
\varphi(\lambda)=\int_{\dR}\frac{d\sigma(t)}{t-\lambda}
\]
is a Nevanlinna function and 
\begin{equation}\label{NevHam}
\varphi(\lambda)=-\frac{s_0}{\lambda}-\frac{s_1}{\lambda^2}-\dots-\frac{s_{2n}}{\lambda^{2n+1}}+
o\left(\frac{1}{\lambda^{2n+1}}\right), \quad \lambda=iy, \quad y\to\infty,
\end{equation} 
for any $n$. Moreover, the Hamburger-Nevanlinna theorem (see \cite[Chapter 3, Section 2]{A65} or \cite[Proposition 4.13]{Simon98}) says that the classical Hamburger moment problem is equivalent to finding a Nevanlinna function with the property \eqref{NevHam} for all nonnegative integer $n$. To solve this equivalent problem one can apply a step-by-step algorithm similar to the Schur algorithm, which for a given $\varphi_0=\varphi$ gives a sequence of Nevanlinna functions 
\[
\varphi_j(\lambda)=-\frac{1}{\lambda-a_j+b_j^2\varphi_{j+1}(\lambda)}, \quad j=0,1,2 \dots,
\]
where $a_j$ and $b_j$ are the same coefficients as in \eqref{qOPrelations}. Furthermore, as one can see the algorithm in the real line case is a straightforward generalization of Euclid's algorithm to the case of formal Laurent series \cite[Section 5.1]{JT}. Consequently, it leads to a continued fraction
\begin{equation}\label{Jfraction}
\varphi(\lambda)\sim-\frac{1}{\lambda-a_0-\displaystyle{\frac{b_0^2}
{\lambda-a_1-\displaystyle{\frac{b_1^2}{\ddots}}}}}=-\cfr{1}{\lambda-a_0}-
\cfr{b_0^2}{\lambda-a_{1}}-
\cfr{b_1^2}{\lambda-a_{2}}-\dots
\end{equation}
which generates the relation \eqref{qOPrelations} in the standard way \cite[Chapter 1, Section 4]{A65}, \cite[Section 5]{Simon98}. By the way, in what follows we will be using the second form of representing continued fractions since it makes formulas more transparent and shorter. 
 
Summing up we see that we have arrived at the desired interpolation problem and the step-by-step algorithm in the real line case. However, this time the interpolation is at $\infty$ (see \cite[Chapter 3, Section 3.6]{A65}) unlike the unit circle case when the corresponding problem is the multiple interpolation at $0$. This means that although the two problems look somewhat similar, they are different in nature. Indeed, the unit circle case concerns the multiple interpolation at $0$, which belongs to the domain of analyticity of Carath\'eodory functions. But the real line case deals with the multiple interpolation at $\infty$, which belongs to the boundary of the domain of analyticity of Nevanlinna functions. That is, a Nevanlinna function does not have to be analytic at $\infty$ and thus we cannot apply the uniqueness theorem. In turn, this entails that there might be many Nevanlinna functions satisfying \eqref{NevHam}, which leads to the theory of extensions of symmetric Jacobi operators to self-adjoint ones \cite{A65}, \cite{Simon98}. 
To conclude this section let us formulate the following statement, which is well known but perhaps was never worded exactly this way.

\begin{proposition}
The trigonometric and Hamburger moment problems are representatives of a class of interpolation problems, which are called Nevanlinna-Pick problems. Moreover, there are Hamburger moment problems that cannot be restated in the form of trigonometric moment problems. In other words, the OPRL theory is not equivalent to the OPUC theory. 
\end{proposition}

\section{Revisiting the Wall ideas}

As explained in the previous section, the theories of OPUC and OPRL do not line up completely since they correspond to different kinds of interpolation problems. So, one can ask a few natural questions. For instance, what would be the theory corresponding to OPUC on the real line? One of the answers to the question is given by the Szeg\H{o} mapping \cite[Section 13.1]{OPUC2}. However, this answer is not entirely natural for the corresponding interpolation problems. As is known, a simple transformation establishes a one-to-one correspondence between Carath\'eodory and Nevanlinna functions. More precisely, it is clear that if $F$ is a Carath\'eodory function then the function
\[
\varphi(\lambda)=iF(z), \quad z=\frac{i-\lambda}{i+\lambda},
\]
is a Nevanlinna function and in view of this transformation there is a natural relation between interpolation problems in the classes of Carath\'eodory and Nevanlinna functions. Besides, the Cayley transform is also natural for relating unitary and self-adjoint operators. Hence, an instinctive way to answer the question would be through the use of the Cayley transform. It appears that this scheme was realized by H. S. Wall \cite{Wall44}, \cite{Wall46} (see also \cite{Wall48}).  Actually, H. S. Wall developed the idea of representing Schur, Carath\'eodory, and Nevanlinna functions by means of continued fractions and the starting point of that study was the classical Schur algorithm. In this section we reframe the core that lies behind Wall's representations and tailor it to our further needs. 

Let us start by noticing that the first obvious discrepancy between OPUC and OPRL is that the sequence of transformations \eqref{SchurAlg} is not a continued fraction contrary to the real line case. Nevertheless, the first step that was done by H. S. Wall is the observation  
\[
f_n(z)=\frac{\gamma_n + zf_{n+1}(z)}{1+\bar\gamma_n zf_{n+1}(z)}
=\gamma_n+\frac{(1-|\gamma_n|^2)z}{\displaystyle{\overline{\gamma}_nz+\frac{1}{f_{n+1}(z)}}}.
\]
Clearly, such representations can be combined into the following expansion
\begin{equation}\label{WF}
f(z)\sim\gamma_0+\cfr{(1-|\gamma_0|^2)z}{\overline{\gamma}_0z}+
\cfr{1}{\gamma_1}+
\cfr{(1-|\gamma_1|^2)z}{\overline{\gamma}_1z}+\dots.
\end{equation}
Let us stress here again that the structure of this fraction and, hence, the underlying recurrence relations are not suitable for having any operator interpretations. So, to speculate one may say that after ending up \eqref{WF} H. S. Wall decided to see if it was possible to do any better for another class of analytic functions. As we saw already the next one in line would be the class of Carath\'eodory functions. 

In order to get Wall's representation of Carath\'eodory functions it will be convenient to represent linear fractional transformations using $2\times 2$ matrices. Namely, we can follow the notation from \cite[page 33]{OPUC1} and rewrite \eqref{SchurAlg} in the following manner
\begin{equation}\label{MatSA}
\begin{pmatrix}
f_n(z)\\
1
\end{pmatrix}\stackrel{.}{=}
\begin{pmatrix}
z&\gamma_n\\
\overline{\gamma}_nz&1
\end{pmatrix}
\begin{pmatrix}
f_{n+1}(z)\\
1
\end{pmatrix},
\end{equation}
where the symbol $\stackrel{.}{=}$ is used in the sense that
\[
\begin{pmatrix}
a\\
b
\end{pmatrix}\stackrel{.}{=}
\begin{pmatrix}
c\\
d
\end{pmatrix} \Leftrightarrow \frac{a}{b}=\frac{c}{d}.
\]
Next step, is to use the Wall Ansatz, which consists in introducing the function
\begin{equation}\label{WallAnsatz}
\begin{pmatrix}
h_n(z)\\
1
\end{pmatrix}\stackrel{.}{=}
\begin{pmatrix}
-\delta_n&1\\
{\delta}_nz&1
\end{pmatrix}
\begin{pmatrix}
f_{n}(z)\\
1
\end{pmatrix},
\end{equation} 
where the coefficients $\delta_n$ are given by the recurrence relations
\[
\delta_0=1, \quad \delta_n=\frac{\overline{\gamma}_{k-1}-\delta_{k-1}}{1-\gamma_{k-1}\delta_{k-1}}, \quad k=1,2, \dots n.
\] 
Then \eqref{MatSA} becomes
\[
\begin{pmatrix}
h_n(z)\\
1
\end{pmatrix}\stackrel{.}{=}
\begin{pmatrix}
-\delta_n&1\\
{\delta}_nz&1
\end{pmatrix}
\begin{pmatrix}
z&\gamma_n\\
\overline{\gamma}_nz&1
\end{pmatrix}
\begin{pmatrix}
-\delta_{n+1}&1\\
{\delta}_{n+1}z&1
\end{pmatrix}^{-1}
\begin{pmatrix}
h_{n+1}(z)\\
1
\end{pmatrix},
\]
which reduces to
\begin{equation}\label{HalfTr}
\begin{pmatrix}
h_n(z)\\
1
\end{pmatrix}\stackrel{.}{=}
\begin{pmatrix}
0& (z+1)(\delta_n-\overline{\gamma}_n)\\
z(z+1)\frac{\delta_n(1-|\gamma_n|^2)}{1-\gamma_n\delta_n}& \delta_n(z+1)\left(\frac{1-\overline{\gamma}_n\overline{\delta}_n}{1-\gamma_n\delta_n}-z\right)
\end{pmatrix}
\begin{pmatrix}
h_{n+1}(z)\\
1
\end{pmatrix}.
\end{equation}
Since the $(1,1)$-entry of the $2\times 2$ matrix in \eqref{HalfTr} is $0$, the corresponding transformations obviously lead to a continued fraction but before writing it down we can simplify it. To this end, let us notice that due to the definition of $\stackrel{.}{=}$, multiplying each of the entries of the $2\times 2$ matrix by the same non-vanishing expression gives a relation equivalent to the original one. Particularly, if we multiply the matrix by $\overline{\delta}_n(1-\gamma_n\delta_n)/(z+1)$ and take into account that $|\delta_n|=1$, we get
 \begin{equation}\label{MainMTr}
\begin{pmatrix}
h_n(z)\\
1
\end{pmatrix}\stackrel{.}{=}
\begin{pmatrix}
0& |1-\gamma_n\delta_n|^2\\
z{(1-|\gamma_n\delta_n|^2)}& (1-\overline{\gamma}_n\overline{\delta}_n)-(1-\gamma_n\delta_n)z
\end{pmatrix}
\begin{pmatrix}
h_{n+1}(z)\\
1
\end{pmatrix}. 
\end{equation}
Still, the elements of the matrix in the latter relation looks a bit heavy and it's possible to make them easier as was done by H. S. Wall. So, let us introduce two sequence of numbers
\begin{equation}\label{gr}
g_{n+1}=\frac{|1-\gamma_n\delta_n|^2}{2\Rl(1-\gamma_n\delta_n)},\quad
r_{n+1}=-\frac{\Im(1-\gamma_n\delta_n)}{\Rl(1-\gamma_n\delta_n)}, \quad n=0, 1, 2, \dots
\end{equation}
and then divide the $2\times 2$ matrix in \eqref{MainMTr} by $\Rl(1-\gamma_n\delta_n)>0$. This manipulation leads to the equivalent representation of \eqref{MainMTr} 
\begin{equation}\label{MainMTrgr}
\begin{pmatrix}
h_n(z)\\
1
\end{pmatrix}\stackrel{.}{=}
\begin{pmatrix}
0& 2g_{n+1}\\
2(1-g_{n+1})z& (1+ir_{n+1})-(1-ir_{n+1})z
\end{pmatrix}
\begin{pmatrix}
h_{n+1}(z)\\
1
\end{pmatrix}. 
\end{equation}
Next, one can observe that \eqref{gr}, \eqref{SchurCond}, and $|\delta_k|=1$ imply 
\begin{equation}\label{grCond}
0<g_k<1,\quad -\infty<r_k<+\infty, \quad k=1,2,3,\dots.
\end{equation}
Indeed, to see the validity of the first one we need to notice that $\Rl(1-\gamma_n\delta_n)>0$ and 
\[
\begin{split}
1>&|\gamma_n\delta_n|^2=|1-(1-\gamma_n\delta_n)|^2=
\\&1-2\Rl(1-\gamma_n\delta_n)+\Rl^2(1-\gamma_n\delta_n)+\Im^2(1-\gamma_n\delta_n)=
\\&1-2\Rl(1-\gamma_n\delta_n)+|1-\gamma_n\delta_n|^2.
\end{split}
\]
The second inequality \eqref{grCond} is obvious. Now, everything is clean and we 
can consecutively apply the linear fractional transformations \eqref{MainMTrgr}, which eventually gives a continued fraction expansion of $h_0$ 
\begin{equation}\label{CFforh_0}
h_0(z)=\frac{1-f(z)}{1+zf(z)}\sim \cfr{2g_1z}{(1+ir_{1})-(1-ir_{1})z}+
\cfr{4(1-g_{1})g_2z}{(1+ir_{2})-(1-ir_{2})z}+
\dots.
\end{equation}
The structure of the latter continued fraction resembles the continued fraction derived by Ya. L. Geronimus \cite{Ger41} but Geronimus' fraction is different. Besides, the continued fraction \eqref{CFforh_0} is implicitly present in \cite{DG}, where a tridiagonal approach to OPUC was developed in the relation to the Bistritz test and the split Levinson algorithm. 

Switching to the original goal, which is to get a continued fraction representation of Carath\'eodory functions, one can see that in general $h_0$ does not belong to the class of Carath\'eodory functions. Nevertheless, a simple linear fractional transformation sends $h_0$ to $F$ defined by \eqref{Carathe}
\[
F(z)=\frac{1+z}{1-z+2z h_0(z)}.
\]
Therefore, we arrive at the following representation of $F$
\begin{equation}\label{CarFrac}
\Scale[1.05]{F(z)\sim \cfr{1+z}{1-z}+\cfr{4g_1z}{(1+ir_{1})-(1-ir_{1})z}+
\cfr{4(1-g_{1})g_2z}{(1+ir_{2})-(1-ir_{2})z}+
\cfr{4(1-g_{2})g_3z}{(1+ir_{3})-(1-ir_{3})z}+\dots}.
\end{equation}
Besides, the continued fraction \eqref{CarFrac} converges locally uniformly in $\dD$ \cite[(iv) on page 292]{Wall48}, which is essentially a consequence of a convergence result regarding to \eqref{WF} (see \cite[Theorem 77.1]{Wall48}).

Finally, we are in the position to formulate a partial answer to the question posed at the beginning of this section.
\begin{theorem}[Wall's theorem] \label{WallTf} Let $F$ be a Carath\'eodory function corresponding to the Schur parameters $\gamma_0$, $\gamma_1$, $\gamma_2$, \dots. Then the Nevanlinna function $\varphi$ defined via the relation
\begin{equation}\label{CayleyT}
\varphi(\lambda)=iF(z), \quad z=\frac{i-\lambda}{i+\lambda}
\end{equation}
can also be generated with the help of the following continued fraction
\begin{equation}\label{ThforN}
\varphi(\lambda)\sim
-\cfr{1}{\lambda}-\cfr{g_1(\lambda^2+1)}{\lambda-r_{1}}-
\cfr{(1-g_{1})g_2(\lambda^2+1)}{\lambda-r_{2}}-
\cfr{(1-g_{2})g_3(\lambda^2+1)}{\lambda-r_{3}}-\dots,
\end{equation} 
where the numbers $g_k$ and $r_k$ are defined by \eqref{gr}. Conversely, any two sequences of numbers $g_k$ and $r_k$ that obey \eqref{grCond} produce a Nevannlina function $\varphi$ normalized by the condition $\varphi(i)=i$ through \eqref{ThforN}.  
In this case, the Schur parameters of the Carath\'eodory function $F$ defined by \eqref{CayleyT} can be recovered in the following two steps:
\[
u_k:=1-\gamma_k\delta_k=\frac{2g_{k+1}}{1+r_{k+1}^2}-\frac{2g_{k+1}r_{k+1}}{1+r_{k+1}^2}i
\]
and then since $\delta_0=1$ we have 
\[
\gamma_0=1-u_0, \quad \gamma_{k+1}=\frac{u_0u_1\dots u_k}{\overline{u}_0\overline{u}_1\dots \overline{u}_k}(1-u_{k+1}),\quad k=0, 1, 2, \dots.
\]
\end{theorem}
\begin{proof}
To get \eqref{ThforN} from \eqref{CarFrac} is easy. It is just the straightforward substitution of \eqref{CayleyT} into
\eqref{CarFrac}. The rest is a consequence of the convergence result that was mentioned above and simple algebraic manipulations. 
\end{proof}
The reader who is familiar with the types of continued fractions can easily recognize a Thiele fraction in \eqref{ThforN} \cite[Appendix A]{JT}. More precisely, it looks like an even or odd part of a Thiele fraction (see \cite[Sections 2.4.2 and 2.4.3]{JT} for the definitions of even and odd parts of a continued fraction) and more details abou the relation between \eqref{ThforN} and Thiele fractions can be found in \cite{STZh}. In fact, Thiele fractions are associated with an interpolating process. Hence, it is quite natural that we got them as we are solving the interpolation problem, which consists in finding necessary and sufficient conditions on the Taylor coefficients of $\varphi$ at $i$ in order that $\Im\varphi(\lambda) >0$ for $\lambda\in\dC_+$. Saying it differently, \eqref{ThforN} is basically the image of the Schur algorithm under the Cayley transform. 

At the same time, \eqref{ThforN} is a particular case of continued fractions of type $R_{II}$ that were introduced by M. Ismail and D. Masson \cite{IM95} and were shown to generate biorthogonal rational functions. That is, we have some orthogonality behind the scene here but we can also see this in a different way since in a sense the original object is OPUC.
 
\begin{corollary}[Wall's characterization]\label{CharNev}
There is a bijection between pairs of infinite sequence with the property \eqref{grCond} and Nevanlinna functions $\varphi$ normalized by the condition $\varphi(i)=i$ provided that we also consider finite sequences where the last $g_n$ is equal to $1$. In case we have infinitely many $g_k$, the continued fraction \eqref{ThforN} converges locally uniformly in $\dC_+$.  
\end{corollary}
\begin{proof}
The statement is a consequence of Theorem \ref{WallTf} and the corresponding result for Schur parameters \cite{DK03} (see also \cite[Theorem 3.1.3]{OPUC1}).
\end{proof}
\begin{remark} Almost everything from this section is present in \cite{Wall44} or \cite{Wall46} in one way or another. Besides, the corresponding material was also included to the book \cite[Sections 77 and 78]{Wall48}. One of a few modifications done here is the direct use of \eqref{CayleyT}. As a matter of fact, H. S. Wall didn't obtain \eqref{ThforN} immediately from \eqref{CarFrac} and rather used an intermediate class of analytic functions (see \cite[Theorem 78.1]{Wall48}) in order to get his characterization of Nevanlinna functions. Nevertheless, the combination of his steps gives exactly the Cayley transform \eqref{CayleyT}.
\end{remark}

\section{Approximants to the Wall continued fractions}

It is well known that OPRL appear as denominators of approximants to $J$-fractions \eqref{Jfraction} (details can be found in \cite{A65} or \cite{Simon98}). So, in this section we are going to explore the approximants to the continued fraction \eqref{ThforN}, which does resemble \eqref{Jfraction} but corresponds to the multiple interpolation at $i$ instead of the multiple interpolation at $\infty$. 

To begin with, note that we know that  combining the first $n+1$ iterates \eqref{SchurAlg} leads to the following representation of the function $f=f_0$ 
\begin{equation}\label{nIterations}
f(z)=\frac{A_{n}(z)+zB_{n}^*(z)f_{n+1}(z)}{B_{n}(z)+zA_{n}^*(z)f_{n+1}(z)},
\end{equation}
where $A_{n}$, $B_{n}$ are polynomials called Wall polynomials, $A_{n}^*$, $B_{n}^*$ are the reversed polynomials defined by 
\[
A_n^*(z)=z^n\overline{A_n(1/\overline{z})}, \quad B_n^*(z)=z^n\overline{B_n(1/\overline{z})},
\]
and $f_{n+1}(z)$ is the $(n+1)$-th iterate of the Schur algorithm \cite[Section 1.3]{OPUC1}. 
It turns out that the Wall polynomials and the sequence $\{\Phi_n\}_{n=0}^{\infty}$ of OPUC are related via the Pint\'{e}r-Nevai formula \cite{PN}:
\begin{equation*}
\Phi_n(z)=zB_{n-1}^*(z)-A_{n-1}^*(z),\qquad \Phi_n^*(z)=B_{n-1}(z)-zA_{n-1}(z).
\end{equation*}
In fact, there are recurrence formulas for Wall polynomials which can be easily obtained from the matrix interpretation of formula \eqref{nIterations}, that is,
\[
\begin{pmatrix}
zB_n^*(z)&A_n(z)\\
zA_n^*(z)&B_n(z)
\end{pmatrix}
=\begin{pmatrix}
z&\gamma_0\\
\overline{\gamma}_0z&1
\end{pmatrix}
\begin{pmatrix}
z&\gamma_1\\
\overline{\gamma}_1z&1
\end{pmatrix}\dots
\begin{pmatrix}
z&\gamma_n\\
\overline{\gamma}_nz&1
\end{pmatrix}.
\]
Then, the next order of business will be to find the formulas for approximants to the continued fraction \eqref{CarFrac}. To this end, let us first recall that 
\[
\begin{pmatrix}
F(z)\\
1
\end{pmatrix}\stackrel{.}{=}
\begin{pmatrix}
0&1+z\\
2z&1-z
\end{pmatrix}
\begin{pmatrix}
h_{0}(z)\\
1
\end{pmatrix}.
\]
Secondly, it follows from \eqref{nIterations} that 
\[
\begin{pmatrix}
f_0(z)\\
1
\end{pmatrix}\stackrel{.}{=}\begin{pmatrix}
zB_n^*(z)&A_n(z)\\
zA_n^*(z)&B_n(z)
\end{pmatrix}\begin{pmatrix}
f_{n+1}(z)\\
1
\end{pmatrix}
\] 
Next, taking into account \eqref{WallAnsatz} the latter relation reduces to
\[
\begin{pmatrix}
h_0(z)\\
1
\end{pmatrix}\stackrel{.}{=}\begin{pmatrix}
-1&1\\
z&1
\end{pmatrix}\begin{pmatrix}
zB_n^*(z)&A_n(z)\\
zA_n^*(z)&B_n(z)
\end{pmatrix}
\begin{pmatrix}
-\delta_{n+1}&1\\
{\delta}_{n+1}z&1
\end{pmatrix}^{-1}
\begin{pmatrix}
h_{n+1}(z)\\
1
\end{pmatrix}
\]
which further gives
\[
\begin{pmatrix}
F(z)\\
1
\end{pmatrix}\stackrel{.}{=}
\begin{pmatrix}
0&1+z\\
2z&1-z
\end{pmatrix}\begin{pmatrix}
-1&1\\
z&1
\end{pmatrix}\begin{pmatrix}
zB_n^*(z)&A_n(z)\\
zA_n^*(z)&B_n(z)
\end{pmatrix}
\begin{pmatrix}
1&-1\\
-\delta_{n+1}z&-{\delta}_{n+1}
\end{pmatrix}
\begin{pmatrix}
h_{n+1}(z)\\
1
\end{pmatrix}.
\]
Now, introducing 
\[
\begin{split}
W_n(z)&=
\begin{pmatrix}
w_{1,1}^{(n)}(z)&w_{1,2}^{(n)}(z)\\
w_{2,1}^{(n)}(z)&w_{2,2}^{(n)}(z)
\end{pmatrix}
\\&:=
\begin{pmatrix}
z^2+z&1+z\\
-z&2
\end{pmatrix}\begin{pmatrix}
zB_n^*(z)&A_n(z)\\
zA_n^*(z)&B_n(z)
\end{pmatrix}
\begin{pmatrix}
1&-1\\
-\delta_{n+1}z&-{\delta}_{n+1}
\end{pmatrix}
\end{split}
\]
yields
\begin{equation}\label{nIterationsCar}
F(z)=\frac{w_{1,1}^{(n)}(z)+w_{1,2}^{(n)}(z)h_{n+1}(z)}{w_{2,1}^{(n)}(z)+w_{2,2}^{(n)}(z)h_{n+1}(z)},
\end{equation}
which is another form of the representation
 \begin{equation}\label{CarFracTrunc}
\Scale[1.05]{F(z)=\cfr{1+z}{1-z}+\cfr{4g_1z}{(1+ir_{1})-(1-ir_{1})z}+
\dots+
\cfr{4(1-g_{n})g_{n+1}z}{(1+ir_{n+1})-(1-ir_{n+1})z+2(1-g_{n+1})zh_{n+1}(z)}}.
\end{equation}
Hence, we get the formula for the approximants
\[
\frac{w_{1,1}^{(n)}(z)}{w_{2,1}^{(n)}(z)}=
\cfr{1+z}{1-z}+\cfr{4g_1z}{(1+ir_{1})-(1-ir_{1})z}+
\dots+
\cfr{4(1-g_{n})g_{n+1}z}{(1+ir_{n+1})-(1-ir_{n+1})z}
\]
by setting $h_{n+1}=0$ in \eqref{nIterationsCar} and \eqref{CarFracTrunc}.
\begin{remark} It appears that in \cite{CCSV14} the authors considered the following continued fraction 
\[
\frac{1+z-(1-z)F(z)}{2zF(z)}=h_0(z)=\cfr{2g_1}{(1+ir_{1})-(1-ir_{1})z}+
\cfr{4(1-g_{1})g_2z}{(1+ir_{2})-(1-ir_{2})z}+\dots
\]
and in subsequent papers they were developing a theory of the corresponding polynomials. Actually, as one can see from the above reasoning, such a theory is just a veiled theory of OPUC. However, this relation deserves a special attention from the point of view of spectral transformation (see \cite{Zh97} for the terminology and explanations of the importnace to the field) as $h_0$ is a spectral transformation of $F$ and vice versa. 
\end{remark}

At this point we are ready to figure out what is happening for the case of Nevanlinna functions. To do so, one has to make the Cayley transform \eqref{CayleyT}. So, after applying it, the first transformation 
\[
\begin{pmatrix}
0&1+z\\
2z&1-z
\end{pmatrix}
\] 
 reduces to
\[
\cW_0(\lambda)=\begin{pmatrix}
0&-1\\
i-\lambda&\lambda
\end{pmatrix},
\] 
where some simplifications were made in accordance with the fact that $\cW_0$ represents a linear fractional transformation. Then, the transformations \eqref{MainMTrgr} become
\[
\cW_n(\lambda)=
\begin{pmatrix}
0& g_{n}(i+\lambda)\\
(1-g_{n})(i-\lambda)& \lambda-r_{n}
\end{pmatrix}, \quad n=1,2,\dots.
\]  
Clearly, the family $\cW_n$ generates \eqref{ThforN} and also we have
\[
\begin{pmatrix}
\varphi(\lambda)\\
1
\end{pmatrix}\stackrel{.}{=}
\cW_0(\lambda)
\cW_1(\lambda)
\dots
\cW_{n}(\lambda)
\begin{pmatrix}
H_{n}(\lambda)\\
1
\end{pmatrix}, \quad n=0, 1, 2\dots,
\]
where $H_{n}(\lambda)=h_{n}(z)$ with $z=(i-\lambda)/(i+\lambda)$ or, equivalently, 
\begin{equation}\label{ThforNtrunc}
\varphi(\lambda)=
-\cfr{1}{\lambda}-\cfr{g_1(\lambda^2+1)}{\lambda-r_{1}}-
\dots-
\cfr{(1-g_{n-1})g_n(\lambda^2+1)}{\lambda-r_{n}+(1-g_{n})(i-\lambda)H_n(\lambda)}.
\end{equation} 
As a result, we have the following statement.
\begin{theorem}
The transfer matrix
\[
\cW_{[0,n]}(\lambda):=\cW_0(\lambda)
\cW_1(\lambda)
\dots
\cW_{n}(\lambda)
\] 
has the following structure
\[
\cW_{[0,n]}(\lambda)=
\begin{pmatrix}
(1-g_n)(i-\lambda)\cA_{n-1}(\lambda)&\cA_{n}(\lambda)\\
(1-g_n)(i-\lambda)\cB_{n-1}(\lambda)&\cB_{n}(\lambda)
\end{pmatrix},
\]
where the polynomials $\cA_n$ and $\cB_n$ satisfy the recurrence relations
\begin{equation}\label{MPrelations}
\begin{split}
\cA_{n}(\lambda)=(\lambda-r_n)\cA_{n-1}(\lambda)-(1-g_{n-1})g_{n}(\lambda^2+1)\cA_{n-2}(\lambda)\\
\cB_{n}(\lambda)=(\lambda-r_n)\cB_{n-1}(\lambda)-(1-g_{n-1})g_{n}(\lambda^2+1)\cB_{n-2}(\lambda)
\end{split}
\end{equation}
with the initial conditions
\[
\cA_{-1}=0, \quad \cA_{0}=-1, \quad \cB_{-1}=1, \quad \cB_{0}=\lambda.
\]
Besides, the formula
\begin{equation}\label{DmNP}
\varphi(\lambda)=\frac{(1-g_n)\cA_{n-1}(\lambda)(\tau(\lambda)^{-1}+\lambda)-\cA_{n}(\lambda)}{(1-g_n)\cB_{n-1}(\lambda)(\tau(\lambda)^{-1}+\lambda)-\cB_{n}(\lambda)}
\end{equation}
gives a description of all Nevanlinna functions $\varphi$ that have the prescribed $n+1$ Taylor coefficient at $\lambda=i$ in terms of an arbitrary Nevanlinna functions $\tau$ such that 
\begin{equation}\label{FirstI}
\varphi(i)=i.
\end{equation}
In fact, the latter condition is also the first interpolation condition because it is inherited from the unit circle case since we only consider probability measures on $\dT$.
\end{theorem}
\begin{proof}
At first, define $\cW_{[0,n]}$ to be
\[
\cW_{[0,n]}(\lambda)=
\begin{pmatrix}
\cC_{n}(\lambda)&\cA_{n}(\lambda)\\
\cD_{n}(\lambda)&\cB_{n}(\lambda)
\end{pmatrix}.
\] 
Next, the definition of  $\cW_{[0,n]}$  entails the formula
\[
\begin{split}
\begin{pmatrix}
\cC_{n}(\lambda)&\cA_{n}(\lambda)\\
\cD_{n}(\lambda)&\cB_{n}(\lambda)
\end{pmatrix}&=\cW_{[0,n]}(\lambda)=\cW_{[0,n-1]}(\lambda)\cW_{n}(\lambda)\\
&=
\begin{pmatrix}
\cC_{n-1}(\lambda)&\cA_{n-1}(\lambda)\\
\cD_{n-1}(\lambda)&\cB_{n-1}(\lambda)
\end{pmatrix}\begin{pmatrix}
0& g_{n}(i+\lambda)\\
(1-g_{n})(i-\lambda)& \lambda-r_{n}
\end{pmatrix},
\end{split}
\]
which gives 
\[
\cC_{n}(\lambda)=(1-g_n)(i-\lambda)\cA_{n-1}(\lambda), \quad \cD_{n}(\lambda)=(1-g_n)(i-\lambda)\cB_{n-1}(\lambda),
\] 
and then \eqref{MPrelations} follows from the rest of the relations. The initial conditions are the result of the form of $\cW_0$. So, we can now proceed to formula \eqref{DmNP}. As a matter of fact, it is a consequence of \eqref{ThforNtrunc} and the corresponding result for Schur functions \cite{DK03} if we set 
\[
\tau(\lambda)=-\frac{1}{\lambda+(i-\lambda)H_n(\lambda)},
\]
which means that
\[
\tau(\lambda)=-\cfr{1}{\lambda}-
\cfr{g_{n+1}(\lambda^2+1)}{\lambda-r_{n+1}}-
\cfr{(1-g_{n+1})g_{n+2}(\lambda^2+1)}{\lambda-r_{n+2}}-
\dots
\]
and therefore the function $\tau$ could be arbitrary Nevanlinna function verifying \eqref{FirstI} in view of Corollary \ref{CharNev}. 
\end{proof}
\begin{remark} Formulas \eqref{nIterations}, \eqref{nIterationsCar}, and \eqref{DmNP} represent a standard piece of the theory of any truncated Nevanlinna-Pick problem, which is the generality that includes all interpolation problems considered here (for more details see \cite[Chapter 3]{A65}). Those formulas are obtained one from another and, essentially, is just one formula. In particular, we have that
\[
i\frac{w_{1,1}^{(n)}(z)}{w_{2,1}^{(n)}(z)}=\frac{\cA_{n+1}(\lambda)}{\cB_{n+1}(\lambda)}, \quad z=\frac{i-\lambda}{i+\lambda},
\]
which is a way to have a connection between OPUC and the polynomials that we get on the real line.
\end{remark}

\section{The underlying linear pencils}

Here we will consider the recurrence relations \eqref{MPrelations} as a particular case of the theory of linear pencils of tridiagonal matrices that was elaborated in \cite{BDZh}, \cite{Der10}, and \cite{DZh}, which, in turn, had their origin in \cite{Zh99}.

Above all, there is no doubt that the relations
\begin{equation}\label{rec_rel}
u_{j+1}(\lambda)-(\lambda-r_{j})u_{j}(\lambda)+(1-g_{j-1})g_{j}(\lambda^2+1)u_{j-1}(\lambda)=0,
\quad j=0, 1, 2, \dots
\end{equation}
were known to H. S. Wall for they are naturally associated with the fraction \eqref{ThforN} (for example, see \cite[Section 2.1.1]{JT}). To be more clear, \eqref{rec_rel} is exactly the same as the second relation in \eqref{MPrelations} if one uses the following agreement
\[
u_{n}=\cB_{n-1}, \quad r_0=0, \quad g_0=0. 
\]
Evidently, the initial conditions
\[
u_{-1}=0, \quad u_0=1
\]   
guarantee that $u_{n}=\cB_{n-1}$.

At first glance, one my conclude that although \eqref{rec_rel} is a three-term recurrence relation, it looks peculiar and the connection to Jacobi matrices, which is a very powerful tool for OPRL, is unclear. That might be the reason it didn't attract any attention at that time. So, let's take a more careful look at \eqref{rec_rel} and try to reinterpret the relation as a spectral problem. With this thought in mind, one can rewrite \eqref{rec_rel} in the following manner
\[
u_{j+1}(\lambda)+r_{j}u_{j}(\lambda)+(1-g_{j-1})g_{j}u_{j-1}(\lambda)-\lambda u_{j}(\lambda)+(1-g_{j-1})g_{j}\lambda^2u_{j-1}(\lambda)=0,
\] 
which is a quadratic eigenvalue problem and is also known as a quadratic pencil. The explicit operator form
\[
\Scale[0.92]{
\begin{pmatrix}
  r_0 & 1&  &\\
  g_1& r_1& 1&\\
     & (1-g_{1})g_{2}&r_2&\\
     &&&\ddots&    
      \end{pmatrix}
      \begin{pmatrix}
      u_0\\
      u_1\\
      u_2\\
      \vdots
      \end{pmatrix}-\lambda \begin{pmatrix}
      u_0\\
      u_1\\
      u_2\\
      \vdots
      \end{pmatrix}
      +
     \lambda^2\begin{pmatrix}
     0&&&\\
     g_1&0&&\\
     &(1-g_{1})g_{2}&&\\
     &&\ddots&
     \end{pmatrix}
     \begin{pmatrix}
      u_0\\
      u_1\\
      u_2\\
      \vdots
      \end{pmatrix}=0}
\] 
looks messy but it's an operator interpretation of \eqref{rec_rel}. A usual method to deal with quadratic pencils is to reduce them to linear pencils, that is, to spectral problems of the form $(A-\lambda B)u=0$
and this is what we are going to do next. Actually, it would be rather efficient to use specifics of the problem than applying the standard machinery of pencils. Say, repeating the reasoning from \cite{Der10}, one can verify that the recurrence relation~\eqref{rec_rel} can be renormalized to the
following one
\begin{equation}\label{r_rl}
\mathfrak{b}_j(i-\lambda)\widehat{u}_{j+1}-
(\lambda-\mathfrak{a}_j)\widehat{u}_{j}-{\mathfrak{b}}_{j-1}(i+\lambda)\widehat{u}_{j-1}
=0,\quad j=0, 1, 2, \dots,
\end{equation}
where the numbers $\mathfrak{a}_j$ and $\mathfrak{b}_j$
are defined by the formulas
\begin{equation}\label{ab}
\mathfrak{a}_j=r_j,\quad \mathfrak{b}_j=\sqrt{(1-g_{j})g_{j+1}},\quad
j=0, 1, 2, \dots
\end{equation}
and the transformation $u\to\widehat{u}$ has the following form
\begin{equation}\label{wTransform}
\widehat{u}_0=u_0,\quad
\widehat{u}_j={\displaystyle\frac{u_j}{\mathfrak{b}_0\dots
\mathfrak{b}_{j-1}(i-\lambda)^j}},\quad j=1,2,3,\dots.
\end{equation}
Now, it is easy to see that relation~\eqref{r_rl} leads to the linear pencil 
\begin{equation}\label{WallPencil}
(H-\lambda J)\widehat{u}=0,
\end{equation}
 where
\[
H=\left(%
\begin{array}{cccc}
  \mathfrak{a}_0 & i\mathfrak{b}_0 &  &  \\
 -i{\mathfrak{b}}_0 & \mathfrak{a}_1 & i\mathfrak{b}_1 &  \\
      & -i{\mathfrak{b}}_1 & \mathfrak{a}_2 & \ddots \\
      &     & \ddots & \ddots \\
\end{array}%
\right),\quad\
J=\left(%
\begin{array}{cccc}
  1 & \mathfrak{b}_0 &  &  \\
  \mathfrak{b}_0 & 1 & \mathfrak{b}_1 &  \\
      & \mathfrak{b}_1 & 1 & \ddots \\
      &     & \ddots & \ddots \\
\end{array}%
\right)
\]
are Jacobi matrices. Apart from this, there is another way of transforming \eqref{rec_rel} into a linear pencil \cite[Proposition 3]{Zh99} but the method described in \cite[Proposition 3]{Zh99} gives a liner pencil that is formed by two bi-diagonal matrices. 

As a matter of fact, working with linear pencils is a bit more delicate than dealing with ordinary spectral problems. For instance, even if the operators $J$ and $H$ generating the pencil are symmetric, it doesn't mean that we have to expect good spectral properties. That is why we need to check if we can say more about the Jacobi matrices $J$ and $H$.    
\begin{proposition}\label{positiveWall}
 The operator $J$ is self-adjoint and nonnegative, that is,
\begin{equation}\label{Nonnegativity}
(J\xi,\xi)_{\ell^2}\ge 0
\end{equation}
for any finite $x\in \ell^2$, i.e. $\xi=(\xi_0,\xi_1,\dots,\xi_n,0,0\dots)^{\top}$.
\end{proposition}
\begin{proof}
Since $J$ has only real entries, it would be enough to prove \eqref{Nonnegativity} only for vector with real elements. So, let us consider the quadratic form
\begin{equation}\label{jform2}
\left(J\xi,\xi\right)=\xi_0^2+2\mathfrak{b}_0{\xi}_0{\xi}_1+\xi_1^2+2\mathfrak{b}_1{\xi}_1{\xi}_2+\dots+\xi_n^2,
\end{equation}
where
$\xi=({\xi}_0,{\xi}_1,\dots,{\xi}_n)^{\top}\in\dR^{n+1}$.
To complete the proof we need to use one of Wall's theorems on chain sequences. To this end, recall that a sequence $\mathfrak{b}_j^2$ is called a chain sequence (see \cite[Section 19]{Wall48} or \cite[Chapter III, Section 5]{Chi78}) if there exists a sequence $\{g_k\}_{k=0}^{\infty}$ such that
\[
\begin{split}
0\le g_0<1,\quad 0<g_k<1, \quad k=1,2,3\dots\\
\mathfrak{b}_k^2=(1-g_{k-1})g_k, \quad  k=1,2,3\dots.
\end{split}
\]
In other words, our sequence $\mathfrak{b}_j^2$ is definitely a chain sequence. Consequently, \cite[Theorem 20.1]{Wall48} (see also \cite[Chapter III, Sections 5 and 6 ]{Chi78}) brings us to the desired statement. 
\end{proof}
Therefore, we have the nonnegativity of $J$ and it suggests that there are ways to analyze the linear pencil relatively easy. However, it still has to be done a bit more carefully than in the ordinary case. The rough idea is to replace the problem   
$(H-\lambda J)\widehat{u}=0$ with an ordinary one such as $(J^{-1/2}HJ^{-1/2}-\lambda I)\widehat{u}=0$ (see \cite{Der10}) or $(J^{-1}H-\lambda I)\widehat{u}=0$. The task is not evident as there might be pitfalls and later on we will discuss an example with a drastic change in comparison with Jacobi matrices and that only happens for the linear pencil case. 
\begin{remark}
It is noteworthy that Proposition \ref{Nonnegativity} is just another form of one of Wall's results, which means that H. S. Wall has proved that the pencil $H-\lambda J$ is a good object in the sense of the spectral theory. Under this circumstances, it makes perfect sense to refer to the pencils in question as to Wall pencils. 
\end{remark}

Once we established the fact that Wall pencils lead to self-adjoint operators and, thus, have a reasonable spectral theory behind, we can discuss it.  At first, the results from \cite{Der10} are applicable to Wall pencils but under an additional restriction that the measure in the integral representation of the function $\varphi$ is a probability measure. This condition simply means that the underlying pencil can be reduced to a self-adjoint operator. In all other case, one has to handle non-densely defined self-adjoint operators. 

Regarding the behavior of the entries of $H-\lambda J$, we can get a pencil form for the majority of results from the theory of OPUC. In particular, we are ready to formulate the result that is mainly the motivation to write this note.
\begin{theorem}[The Denisov-Rakhmanov Theorem]\label{DR} Let $\varphi$ be a Nevanlinna function of the form 
\[
\varphi(\lambda)=a\lambda+\int_{\dR}\frac{1+t\lambda}{t-\lambda}d\sigma(t),
\]
which is subject to $\varphi(i)=i$.
If $\sigma'>0$ almost everywhere with respect to the Lebesgue measure on $\dR$ then for the corresponding Wall pencil \eqref{WallPencil} we have that
\begin{equation}
\mathfrak{a}_k\to 0, \quad \mathfrak{b}_k\to\frac{1}{2} \quad \text{ as } k\to \infty.
\end{equation}
\end{theorem}
\begin{proof}
The statement is an immediate consequence of the Rakhmanov theorem for the unit circle and the Wall constructions. Indeed, the condition that
 $\sigma'>0$ almost everywhere with respect to the Lebesgue measure on $\dR$ ensures that the corresponding measure $\mu$ on the unit circle is positive almost everywhere on $\dT$. Actually, the measure $\mu$ is determined by $F$, which, in turn, is defined through \eqref{CayleyT}. Besides, the relation between $\sigma$ and $\mu$ is explicit and can be found in  \cite[Section 59]{AG}. Next, it follows from the Rakhmanov theorem (for example, see \cite[Corollary 9.1.11]{OPUC2}) 
that
\[
\lim_{k\to\infty}\gamma_k=0.
\] 
 The rest is immediate from formulas \eqref{gr} and \eqref{ab}. 
\end{proof}

As one might see the situation is a bit unusual because the entries of the pencil are uniformly bounded but the corresponding measure is supported on the whole real line, which never happens for OPRL. The trick with the support is hidden in the property of $J$, which has no bounded inverse. That is, when we reduce the pencil spectral problem to an ordinary one we have to deal with 
the operator $J^{-1/2}HJ^{-1/2}$, which is not bounded and, in the general situation we have here, doesn't have to be densely defined (see \cite{Der10}).

Note that Theorem \ref{DR} is just another way to look at the spectral theory of OPUC but it does lead to further insights in the theory of linear pencils. Another interesting observation can be made. To this end, let us recall that the results from \cite{Der10} and \cite{DZh} together with the Wall theory basically read that any Nevanlinna function 
\[
\varphi(\lambda)=a\lambda+b+\int_{\dR}\frac{1+t\lambda}{t-\lambda}d\sigma(t),
\]
admits the continued fraction representation
\begin{equation}\label{ContF}
\varphi(\lambda)=-\cfr{1}{a^{(2)}_0\lambda-a_0^{(1)}}-
    \cfr{b_0^2(\lambda-z_0)(\lambda-\overline{z}_0)}{a^{(2)}_1\lambda-a_1^{(1)}}-
    \cfr{b_{1}^2(\lambda-z_1)(\lambda-\overline{z}_1)}{a^{(2)}_2\lambda-a_2^{(1)}}-
    \dots,
\end{equation}
where $z_0$, $z_1$, $\dots$  are some given numbers from $\dC_+$, that is, they are the interpolation nodes. Besides, this continued fraction and the theory in \cite{Der10} and \cite{DZh} are based on a different renormalization of the entries. Still, it is also a result of the Schur algorithm but to obtain \eqref{ContF} one needs to change the point at which the Schur transformation is performed. In addition, as one can see from the findings in \cite{Der10} and \cite{DZh} the entries of \eqref{ContF} depend on the respective interpolation node continuously. Therefore, it seems plausible that the following statement is true.  

\begin{conjecture} Let $\sigma$ be the measure corresponding to $\varphi$. If $\sigma'>0$ almost everywhere with respect to the Lebesgue measure on $\dR$ and $z_k\to i$ as $k\to\infty$ then there exists a renormalization of the coefficients of \eqref{ContF} such that
\[
a_k^{(1)}\to 0,\quad a^{(2)}_k\to 1,\quad b_{k}\to\frac{1}{2}\quad \text{ as } k\to \infty.
\]
\end{conjecture}

Furthermore, one can easily reformulate the Szeg\H{o} theorem in terms of Wall pencils and then make a similar conjecture and that can practically be done for the majority of the results from \cite{OPUC1} and \cite{OPUC2}.  

\section{The reference measure and other related distributions}

Now is the time to consider examples. The first one that we have already encountered is the reference measure, that is,  the limiting pencil from Theorem \ref{DR}. Since the essence of that theorem comes from OPUC, i.e., from the relation 
\[
\lim_{k\to\infty}\gamma_k=0,
\] 
we just need to find the Wall pencil that corresponds to the case
\[
\gamma_k=0, \quad k=0,1,2,\dots.
\]
Next, \eqref{gr} and \eqref{ab} yield
\[
\left(%
\begin{array}{cccc}
  0 & \frac{1}{\sqrt{2}}i &  &  \\
  -\frac{1}{\sqrt{2}}i & 0 & \frac{1}{2}i &  \\
      & -\frac{1}{2}i & 0 & \ddots \\
      &     & \ddots & \ddots \\
\end{array}%
\right)-\lambda \left(%
\begin{array}{cccc}
 1 & \frac{1}{\sqrt{2}} &  &  \\
  \frac{1}{\sqrt{2}} & 1 & \frac{1}{2} &  \\
      & \frac{1}{2} & 1 & \ddots \\
      &     & \ddots & \ddots \\
\end{array}%
\right).
\]
and the underlying Nevanlinna function has the following continued fraction representation
\[
m_0(\lambda)=
-\frac{1}{\lambda-\displaystyle{\frac{\frac{1}{2}(\lambda^2+1)}{{\lambda}-
\displaystyle{\frac{\frac{1}{4}(\lambda^2+1)}{{\ddots}}}}}}.
\]
Besides, since $f(z)\equiv 0$ in $\dD$ is the function that gives the sequence of Schur parameters that are equal to zero,  we can easily get that $m_0(\lambda)=i$ in $\dC_+$, which can be extended to the lower half-plane
\[
m_0(\lambda)=\begin{cases} i &\mbox{if } \lambda\in\dC_+ \\ 
-i & \mbox{if } \lambda\in\dC_- \end{cases}. 
\] 
As is known, $m_0$ has the following integral representation 
\[
m_0(\lambda)=\int_{\dR}\frac{1+t\lambda}{t-\lambda}\frac{dt}{1+t^2}.
\]
In a sense, this example is the simplest one in the theory of Wall pencils but has a feature showing a difference between Wall pencils and Jacobi matrices. Namely, one of the most powerful tools in the theory of Jacobi matrices is the $m$-function (aka Weyl function). Now, if one thinks in terms of $m$-functions of Jacobi matrices then one would define the $m$-function of a Wall pencil by the following formula
\[
m_0(\lambda)=((H_0-\lambda J_0 )^{-1}e,e)_{\ell^2}, \quad e=(1,0,\dots)^\top,
\]
which we would expect to hold for all $\lambda\in\dC_+$ as we are dealing with the self-adjoint case (there is only one solution of the corresponding interpolation problem).
Thus, we would also have
\[
m_0(\lambda)=\frac{1}{\lambda}((1/\lambda-(H_0-\lambda_0 J_0)^{-1}J_0)^{-1}(H_0-\lambda_0J_0)^{-1}e,e)_{\ell^2},
\] 
where $\lambda_0\in\dC_+$ is some fixed point. The latter would imply that
\[
m_0(\lambda)\to 0, \quad \lambda=iy, \quad y\to\infty,
\]
which is impossible in our case. This dissimilarity means that we cannot define $m$-functions of Wall pencils in the same way as it is done for Jacobi matrices (cf. \cite{Der10}). For instance, it could be done in a way similar to the way one defines Weyl functions for differential operators. Say, to this end one could use boundary triplets and abstract Weyl functions \cite{DM95}. Let us stress that although Jacobi matrices and Wall pencils are generated by different interpolation problems, Wall pencils are more generic as they are in one-to-one corresponds with the entire class of Nevanlinna functions, which is not the case for Jacobi matrices. 

In principal, any explicit example of OPUC can be transformed to a Wall pencil, whose entries and the generating measure can be computed. However, looking at the integral representation of $m_0$ one sees the Cauchy distribution in there and it leads to certain ideas related to the mechanism of the transformation. Actually, it turns out that one can easily construct a family of examples based on the Cauchy distribution. As a matter of fact, \cite{BO} deals with the pseudo-Jacobi ensemble, which is based on the Cauchy distribution, and some formulas from \cite{BO} show the presence of functions that would be appropriate to consider in the framework of this note. So, following \cite{BO} 
let us recall that the Gauss hypergeometric function is a function defined via the series
\[
{{}_2F_1\left(\left.\begin{array}{c}a,b\\ c \end{array}\right|z\right)=
1+\frac{a b}{c}\frac{z}{1!}+\frac{a(a+1)b(b+1)}{c(c+1)}\frac{z^2}{2!}
+\dots,}
\]
where $z$ is an independent variable and $a$, $b$, and $c$ are complex parameters. Clearly, if either $a$ or $b$ is a negative integer then ${}_2F_1\left(\left.\begin{array}{c}a,b\\ c \end{array}\right|z\right)$ is a polynomial. Therefore, setting $a=-n$ to be a nonpositive integer and $z=1/(1+ix)$ we see that the function
\[
R_n(x)={}_2F_1\left(\left.\begin{array}{c}-n,b\\ c \end{array}\right| \frac{2}{1+ix}\right)
\]
is a rational function in a new variable $x$. Next, using the contiguous relation \cite[Section 2.5]{AAR99}
\[
\Scale[0.85]{
a(1-z){}_2F_1\left(\left.\begin{array}{c}a+1,b\\ c \end{array}\right| z\right)+
(c-2a-(b-a)z){}_2F_1\left(\left.\begin{array}{c}a,b\\ c \end{array}\right| z\right)-
(c-a){}_2F_1\left(\left.\begin{array}{c}a-1,b\\ c \end{array}\right| z\right)=0}
\]  
for the function $R_n$, one gets
\[
-n\left(1-\frac{2}{1+ix}\right)R_{n-1}(x)+\left(c+2n-(b+n)\frac{2}{1+ix}\right)R_n(x)-(b+n)R_{n+1}(x)=0,
\]
which reduces to the following recurrence relation
\begin{equation}\label{CauchyCom}
n(x+i)R_{n-1}(x)+((c-2b)i-x(c+2n))R_n(x)+(x-i)(b+n)R_{n+1}(x)=0.
\end{equation}
As was shown in \cite{IM95} (see also \cite{Zh99}) if a system satisfies a relation of the type \eqref{CauchyCom}, then it is a system of orthogonal rational functions. That is, the rational functions $R_n$ are orthogonal with respect to some functional. As a matter of fact, \eqref{CauchyCom} is far beyond the scope of the current paper since its coefficients are generally complex numbers with no symmetry. However, if we restrict ourselves to the case
\[
c=2b, \quad b=s\in\dR,
\]
we get the following statement for the functions $R_n(x)=R_n(x,s)$. 
\begin{proposition}
The rational functions $R_n(x,s)$ satisfy the following doubly spectral relations
\begin{equation}\label{CauchyNonSym1}
n(x+i)R_{n-1}(x,s)+2x(s+n)R_n(x,s)+(2s+n)(x-i)R_{n+1}(x,s)=0.
\end{equation}
This means that if $s$ is fixed then \eqref{CauchyNonSym1} as an $x$-relation is easily reducible to \eqref{r_rl}, which is equivalent to a recurrence relation of type $R_{II}$ introduced in \cite{IM95}, and if $x$ is fixed then  \eqref{CauchyNonSym1} is an $R_I$-type recurrence relation in $s$ introduced in \cite{IM95} as well. 

Consequently, there exist a functional in $x$ and a functional in $s$ such that the rational functions $R_n(x,s)$ form an orthogonal system in $s$ and an orthogonal system in $x$.  
\end{proposition}
\begin{proof}
The proof is just the application of Favard's type results from \cite{IM95} to \eqref{CauchyNonSym1}, i.e. we first consider \eqref{CauchyNonSym1} as a relation in $x$ and get the functional in $x$. Then, we look at  \eqref{CauchyNonSym1} as a relation in $s$. 
\end{proof}
\begin{remark}
The polynomials $(x-i)^nR_n$ appear in \cite{BO} and it is their large-$n$ behavior that essentially matters for the pseudo-Jacobi ensemble.
\end{remark}
We can also make \eqref{CauchyNonSym1} symmetric, which can be done by introducing new rational functions
\[
C_n(x)=\sqrt{1^{-1}\frac{1}{2(1+s)}\dots \left(\frac{2s+n-1}{2(n-1+s)}\right)^{-1}\frac{n}{2(n+s)}}R_n(x).
\]
For the new rational functions, relation \eqref{CauchyNonSym1} reads
\begin{equation}\label{CauchySym}
\small
\sqrt{\frac{(n-1+2s)n}{4(n-1+s)(n+s)}}(x+i)C_{n-1}(x)+xC_n(x)+
\sqrt{\frac{(n+2s)(n+1)}{4(n+s)(n+1+s)}}(x-i)C_{n+1}(x)=0.
\end{equation}
\normalsize
One of good things that come from \eqref{CauchySym} is that the zeroes of $C_n$ are real (say, the results of \cite{DZh} or \cite{Der10} can be applied in this case). Besides, one can also notice that for $C_n$ we have
\begin{equation}\label{chain}
r_{k-1}=0,\quad g_k=\frac{k}{2k+2s},\quad k=1,2,3,\dots.
\end{equation}
which shows that the example fits into the Wall theory as long as $0<g_k<1$ for all $k$, that is, when 
\[
s>-\frac{1}{2}.
\]
However, if $s>1/2$ then we can explicitly write the measure of orthogonality by making a connection to pseudo-Jacobi polynomials, which are also called Routh-Romanovski polynomials.
\begin{proposition} If $s>1/2$ then 
\[
\int_{\dR}C_n(t)\frac{1}{(t-i)^k}\frac{dt}{(1+t^2)^s}=0, \quad k=1,2,\dots, n.
\]
\end{proposition}
\begin{proof}
Without loss of generality, we will show the orthogonality for $R_n$ rather than for $C_n$. At first, we notice that $R_n$ is a finite sum and the summation can be reversed in order to get the following
\[
\begin{split}
R_n(x)&={}_2F_1\left(\left.\begin{array}{c}-n,s\\ 2s \end{array}\right| \frac{2}{1+ix}\right)\\
&=\left(\frac{-2}{1+ix}\right)^n\frac{(s)_n}{(2s)_n}{}_2F_1\left(\left.\begin{array}{c}-n,1-n-2s\\ 1-n-s \end{array}\right| \frac{1+ix}{2}\right),
\end{split}
\]
where $(a)_n$ is the Pochhammer  symbol, that is, $(a)_n=a(a+1)\dots(a+n-1)$.
Recalling that Jacobi polynomials are defined by
\[
P_n^{(\alpha,\beta)}(x)=\frac{(\alpha+1)_n}{n!}{}_2F_1\left(\left.\begin{array}{c}-n,n+\alpha+\beta+1\\ \alpha+1 \end{array}\right| \frac{1-x}{2}\right)
\]
we get that
\[
R_n(x)= \frac{c_n}{(1+ix)^n}P_n^{(-s-n,-s-n)}(-ix).
\]
Next, according to \cite{As87} we have
\[
\int_{\dR}P_n^{(-n-s,-n-s)}(ix)(1+ix)^m\frac{dx}{(1+x^2)^{s+n}}=0, \quad m=0,1,2\dots n-1,
\]
which clearly leads to the desired result.
\end{proof}
As one might notice, the theory related to the example may easily go beyond the condition $s>-\frac{1}{2}$. Say, if $s$ is not a negative integer, then we can make the inverse Wall transformation, which will return a sequence of Schur parameters. Then the formulas that determine $\gamma_k$ from $r_k$ and $g_k$ show that there could be only a finite number of Schur parameters that lie outside of the closed unit disc. Such situations are feseable to understand and, recently, it has been shown in \cite{DS16} that the OPUC techniques can still work in such nonclassical cases. 

\medskip

\noindent\textbf{Acknowledgements.} I'd like to thank Alexei Zhedanov, who has been feeding me with the information on the theory underlying $R_{II}$ recurrence relations for years and who, after reading the manuscript, provided me with a few comments that helped to improve the presentation of the paper. Besides, I'm very grateful to my wife, Anastasiia, who read the manuscript and found an enormous amount of typos and inaccuracies.

\end{document}